\documentclass[a4paper]{article}

\usepackage[english]{babel}

\usepackage[a4paper,top=2cm,bottom=2cm,left=3cm,right=3cm,marginparwidth=1.75cm]{geometry}

\usepackage{amsmath}
\usepackage{graphicx}
\usepackage[colorlinks=true, allcolors=blue]{hyperref}
\usepackage{amssymb}
\usepackage{dsfont}
\usepackage{a4wide}
\usepackage[utf8]{inputenc}
\usepackage[T1]{fontenc}
\usepackage{braket}
\usepackage{color}
\usepackage[normalem]{ulem}
\usepackage{verbatim}

\usepackage{enumerate}
\usepackage{enumitem}

\usepackage{amsmath, amssymb, amstext, amsthm, dsfont}

\usepackage[arrow, matrix, curve]{xy}
\usepackage{slashed}
\usepackage{makeidx}

\makeindex
\newcommand{\N}{\mathbb{N}}
\newcommand{\Z}{\mathbb{Z}}

\newcommand{\R}{\mathbb{R}}

\theoremstyle{plain}
\newtheorem{theorem}{Theorem}

\newtheorem{lemma}{Lemma}[section]

\newtheorem{proposition}[lemma]{Proposition}
\newtheorem{remark}[lemma]{Remark}

\newtheorem{definition}[lemma]{Definition}
\theoremstyle{definition}
\newtheorem{example}[lemma]{Example}

\title{A phase transition for a spatial host-parasite model with extreme host immunities on $\mathbb{Z}^d$ and $\mathbb{T}_d$.}
\author{\textsc{Sascha Franck}\footnote{University of G\"ottingen, Institute for Math. Stochastics, Goldschmidtstr.~7, 37077 G\"ottingen, Germany, \texttt{sascha.franck@uni-goettingen.de}}}

\begin{document}
\maketitle
\begin{abstract}
  We investigate a model of a parasite population invading spatially distributed immobile hosts on a graph, which is a modification of the frog model. Each host has an unbreakable immunity against infection with a certain probability $1-p$ and parasites move as simple symmetric random walks attempting to infect any host they encounter and subsequently reproduce themselves. We show that, on $\Z^d$ with $d\ge 2$ and the $d$-regular tree $\mathbb{T}_d$ with $d\ge 3$, the survival probability of parasites exhibits a phase transition at a critical value of $p_c\in(0,1)$. Also, we show that the value of $p_c$ is, in general, not a monotonic function of the underlying graph. Finally, we show that on quasi-vertex-transitive graphs, with probability $1$, a fixed vertex is only visited finitely often by a parasite under mild assumptions on the offspring distribution of parasites.
\end{abstract}
{\footnotesize\textit{2020 Mathematics Subject Classification --} 60K35; 60J80; 92D25; 92D30.

\section{Introduction}
In this work, we investigate the spread of parasites in a spatially distributed host population on different graphs $G= (V,E)$. The main focus will be on the integer lattice $\Z^d$ for $d\ge 2$ and the $d$-regular tree $\mathbb{T}_d$ for $d\ge 3$. By abuse of notation, we denote by $\mathbf{0}$ a distinguished vertex in any graph $G$, which only for $\Z^d$ will actually be the origin. In this model, each vertex $v\in V$ is inhabited by a host, and we assume that initially the host at $\mathbf{0}$ is infected, and a random number of parasites, distributed as some random variable $A$, is placed at $\mathbf{0}$. Parasites move on $G$ according to symmetric simple random walks in discrete time and, when they jump onto a host, attempt to infect this host and reproduce themselves. As hosts often have an immune response against infections, we assume that a host with probability $1-p\in(0,1]$ is completely immune (immune for short) to infection; that is, whenever a parasite tries to infect such a host, parasite reproduction is prevented and the parasite is killed. If the host is not immune, it is called susceptible. In this case, an attacking parasite kills the host, the attacking parasite reproduces (and dies afterwards), and sets free a random number of offspring, distributed as $A$ and independent of everything else. For simplicity we assume that hosts do not reproduce.
We note here that the infecting parasite also dies at a successful infection, and we allow for $0$ offspring to be produced at an infection.

In this paper we investigate the survival probability of the parasite population and show that there is a phase transition for a positive survival probability in the parameter $p$ on both $\Z^d$ and $\mathbb{T}_d$. Also, we investigate recurrence to the origin and show that this does not happen on any graph if the offspring distribution has a finite mean. We will call our model the Spatial Infection Model with host Immunity, or SIMI for shorthand.
\cite{PF2025} make a more thorough analysis of a similar model on $\Z$, which we also call SIMI, where hosts can lose their immunity after getting attacked a random number of times.
Our model generalises the frog model that was introduced by \cite{T1999}. The frog model is a classical branching system on some graph $G=(V,E)$, which involves two types of particles that are usually called \textit{active frogs} and \textit{sleeping frogs}. In the frog model, initially, there are sleeping frogs on the vertices of $G$ and one active frog on some distinguished site $\mathbf{0}\in V$. Sleeping frogs do not move, and active frogs move as simple symmetric random walks on $G$, waking up all sleeping frogs that they encounter and turning them into active frogs as well. The active frogs correspond to parasites in our model, and the sleeping frogs on a vertex $v\in V$ correspond to the offspring that are produced in our model after the host at $v$ gets infected. The frog model coincides with the SIMI for the case that $p=1$ and that there is almost surely at least one offspring produced after a successful infection.

It was shown by \cite{A2001} that the frog model on $\Z^d$ satisfies a shape theorem for the set $S_d(n)$ of vertices visited by some active frog up to time $n$. They showed that, if an i.i.d.~amount of sleeping frogs is placed on each vertex, there is a convex deterministic set $\mathcal{A}_d$, such that for any $\varepsilon\in(0,1)$ we have
\[
(1-\varepsilon)\mathcal{A}_d\subset \frac{S_d(n)}{n}\subset(1+\varepsilon)\mathcal{A}_d ~\text{ almost surely for all }n \text{ large enough.}
\]
We will use this result to show our main result, Theorem \ref{thm:crit_nontriv}. The proof relies on coupling with a supercritical site percolation and using that, due to the shape theorem, if all hosts in a large region are susceptible, then with high probability an infection starting in that region will reach any vertex in a suitable subregion after a linear amount of steps. Making the region big enough and then $p$ close to $1$ will allow us to conclude that the coupled site percolation is supercritical and show the positive survival probability.

Next to the trivial survival of the frogs in the classical model, another main difference from a mathematical standpoint is the following. In the classical frog model, using a collection of independent simple symmetric random walks to assign each frog its trajectory after waking up yields an intuitive way to couple initial configurations in a monotone way. However, doing this in our model will not be a monotone coupling if we allow for $0$ offspring to be produced after a successful infection. This happens because when and where a specific parasite dies depends on the location of hosts that are still alive and will produce $0$ offspring. A concrete example of this phenomenon can be found in Example \ref{ex:notmon}.

\cite{A2002a} introduced and studied a different kind of death mechanic for the frog model. In their model, each time a frog takes a step, it dies with probability $1-\widetilde{p}\in [0,1)$. In this setting the death mechanic only depends on $\widetilde{p}$, in contrast to our model where it also depends on $\mathbb{P}(A=0)$. Hence, it is directly clear that the survival probability of frogs is monotone in the parameter $\widetilde{p}$.

In another work \cite{PF2025} we investigate a generalisation of the model in the current paper on the graph $\Z$, in which hosts can lose their immunity and get infected after being attacked a random number of times. Under certain moment assumptions on the amount of times a host needs to be attacked before infection, we show that the spread of the parasite population is occurring at linear speed. A similar model as in \cite{PF2025} with hosts that lose their immunity after a random number of attacks but without the killing of parasites was also introduced by \cite{J2023} and studied on infinite $d$-ary trees.
\section{Main results}\label{sec:mainres}
In this section we present the main results of this work. We will establish a phase transition for the probability of survival of the parasite population in the parameter $p$. Also, we provide an example to show that the critical parameter is not a monotonic function of the underlying graph. Finally, we investigate recurrence to the origin and show that under some mild assumptions on the offspring distribution, recurrence cannot occur on any transitive graph.
\subsection{Survival of the parasite population}
In this section we will, for any graph $G$ and offspring distributed as $A$, investigate the value of the critical parameter $p_c(G,A)$, which is defined as
\[
 \inf\left\{p\in(0,1]:\mathbb{P}\left(\begin{matrix}\text{The parasite population in the SIMI on }G\\\text{with offspring distribution }A \text{ and}\\\text{susceptible hosts appearing with probability }p\\\text{ survives forever.}\end{matrix}\right)> 0\right\}.
\]
In Definition \ref{def:crit_par} we will give a formal definition of $p_c(G,A)$ after constructing the process and show in Lemma \ref{Lem:survdom} that $p_c(G,A)$ is a critical parameter, in the sense that for $p>p_c(G,A)$ the parasite population on $G$ survives with positive probability and for $p<p_c(G,A)$ it dies out almost surely.
First, we note that on a finite graph $G=(V,E)$ the critical parameter $p_c(G,A)$ is trivially $0$ for any offspring distribution, because with probability $(\mathbb{P}(A\ge 1)p)^{\vert V\vert}>0$ there is no immune host at all, and every infection produces at least one offspring. Hence, in the following, we always assume that $G$ is an infinite graph.

Our first result states that on any (infinite) graph this critical parameter is positive for any $A$ that has finite expectation.
    \begin{theorem}\label{thm:trivpos}
    For any infinite graph $G$ we have $p_c(G,A) \ge \min\left\{1,\frac{1}{\mathbb{E}[A]}\right\}$,  where $\frac{1}{\infty} := 0$.
    \end{theorem}
    \begin{proof}
Let $p < \min\left\{1, \frac{1}{\mathbb{E}[A]}\right\}$. On any infinitely large graph, the number of sites visited by a simple symmetric random walk is infinite. Since completely immune hosts remain in the system forever, each parasite will eventually hit a susceptible host or hit a completely immune host, and surviving through parasites that walk through empty space and never meet a host (and thus do not die) is not possible. Say some parasite jumps onto a host at a vertex $v\in V$. Then, if it is the first time that this host at $v$ is attacked by a parasite, with probability $p$ the host is susceptible and the parasite generates offspring distributed as $A$, and with probability $1-p$ the host is completely immune and no offspring are generated. If the host at vertex $v$ was already visited by a different parasite, then the host has to be a completely immune host, or else it would have been killed by that other parasite. Hence, in this case no offspring are generated. In particular, ignoring if it was the first time $v$ was visited, we obtain that the offspring generated by a parasite are stochastically dominated by a variable that with probability $p$ is distributed as $A$ and is $0$ with probability $1-p$. Because a susceptible host can only be used once for reproduction, this means that the limit of the number of living parasites as $n\to\infty$ is stochastically dominated by the limit $\lim_{n\to\infty} \xi_n$,
    where $(\xi_n)_{n\ge0}$ is a Bienaymé-Galton-Watson branching process with offspring distributed as $A$ with probability $p$ and $0$ with probability $1-p$. Since this Bienaymé-Galton-Watson branching process has mean offspring $p\mathbb{E}[A]<1$, the claim follows.
    \end{proof}
    We can compare our model to a site percolation on any graph $G$, which is a classical model in probability theory (see e.g., \cite{G1999} for an introduction). In this model, each vertex on a graph $G$, independently of everything else, is either open with probability $p\in[0,1]$ or closed with probability $1-p$. An important quantity is the critical parameter $\widetilde{p}_c(G)$, which is the supremum over all $p\in[0,1]$, such that with probability $1$ there is no infinite, connected path of open vertices starting at some distinguished vertex $\mathbf{0}$. A classic result is that on graphs with bounded degree, the critical parameter is always positive, which yields the analogue result for the SIMI.
    \begin{theorem}\label{thm:trivpos2}
For any graph $G$ and any offspring distribution $A$, we have that $p_c(G,A) \ge \widetilde{p}_c(G)$, where $\widetilde{p}_c(G)$ is the critical site percolation parameter on $G$. In particular, if each vertex has a degree bounded by some $\Delta <\infty$, then $p_c(G,A) \ge 1/(\Delta-1)$ and thus $p_c(\Z,A) = 1$.
\end{theorem}
\begin{proof}
 We declare a site percolation on $G$ by saying a site is open if it is inhabited by a susceptible host. We recall that in a site percolation on $G$ with parameter $p<\widetilde{p}_c(G)$, almost surely there is no infinite path of open vertices.  In particular, this means that for $p<\widetilde{p}_c(G)$, almost surely there is some finite connected (as a subgraph of $G$) set $\mathcal{J}\subset V$ that contains $\mathbf{0}$, such that every $x\in\mathcal{J}\setminus\{\mathbf{0}\}$ is inhabited by a susceptible host and every $x\in V\setminus\mathcal{J}$, for which there is a $y\in \mathcal{J}$ such that $\{x,y\}\in E$, is inhabited by an immune host. Since every neighbour of $\mathcal{J}$ is immune, a parasite dies when it attempts to leave the finite set $\mathcal{J}$. Because only finitely many parasites get generated inside the finite set $\mathcal{J}$, the parasite population dies out in finite time. The second and third claims are a direct consequence of \cite[Theorem 1.33]{G1999}.
\end{proof}
\begin{remark}
 We note that Theorem \ref{thm:trivpos2} in particular makes no assumptions on the distribution of $A$. This distinguishes the SIMI from the similar frog model with death introduced by \cite{A2002a}, for which a positive critical parameter on graphs with bounded degree was only shown under additional (mild) assumptions on the offspring distribution (cf. \cite[Theorem 1.3]{A2002a} and \cite[Theorem 2.1]{L2020}).
\end{remark}

\subsection{Survival on \texorpdfstring{$\Z^d$}{Zd}}
Our main theorems will establish that $p_c(\Z^d,A) < 1$ for all $d\ge 2$ under mild conditions on $A$.
Our main theorem states the following.
  \begin{theorem}\label{thm:crit_nontriv}
    For any $d\ge 2$ and $A$ such that $A\ge 1$ almost surely and $\mathbb{E}[A] > 1$, we have
    \[
    p_c(\Z^d,A) < 1.
    \]
    \end{theorem}
    The proof relies on coupling with a supercritical site percolation and that for $p=1$ the SIMI coincides with the frog model, for which a shape theorem is known. The details will be given in section \ref{subsec:proof_Zd}.

As a corollary we immediately obtain the following result for general $A$.
\begin{theorem}\label{thm:crit_genA_Z}
  Let $d\ge 2$ and $A$ be such that $\mathbb{E}[A] > 1$. Let $\widetilde{A}$ be a random variable that is distributed as $A$ conditioned to be at least $1$, and suppose $\mathbb{P}(A = 0) < 1 - p_c(\Z^d,\widetilde{A})$. Then also $p_c(\Z^d,A) < 1$.
\end{theorem}
\subsection{Survival on \texorpdfstring{$\mathbb{T}_d$}{Td}}
Similar to $\Z^d$, we obtain a phase transition whenever at least one offspring is generated at an infection.
\begin{theorem}\label{thm:Tdphase}
  Let $A\ge 1$ almost surely and $\mathbb{E}[A] > 1$ and $d\ge 3$. Then $p_c(\mathbb{T}_d,A) < 1$.
\end{theorem}
The proof is similar to the approach of \cite[Theorem 1.5]{A2002a} and relies on coupling with a supercritical Bienaymé-Galton-Watson branching process. The details will be given in section \ref{subsec:proof_Td}.

Using the same argument as on $\Z^d$, we will obtain the following.
\begin{theorem}\label{thm:crit_genA_T}
  Let $d\ge 3$ and $A$ such that $\mathbb{E}[A] > 1$. Let $\widetilde{A}$ be a random variable that is distributed as $A$ conditioned to be at least $1$, and suppose $\mathbb{P}(A = 0) < 1 - p_c(\mathbb{T}_d,\widetilde{A})$. Then also $p_c(\mathbb{T}_d,A) < 1$.
\end{theorem}
For a large degree $d$, parasites will jump away from the root onto a site still inhabited by a host most of the time, and hence, the survival will approximately be like that of a Bienaymé-Galton-Watson branching process. Precisely, we obtain the following asymptotic.
\begin{theorem}\label{thm:Tdasympt}
  Suppose that $1<\mathbb{E}[A]\le \infty$, then $\lim_{d\to\infty}p_c(\mathbb{T}_d,A) = \frac{1}{\mathbb{E}[A]}$.
\end{theorem}
Again, the proof follows a similar approach to \cite[Theorem 1.7]{A2002a}, relying on coupling with a Bienaymé-Galton-Watson branching process, and will be given in section \ref{subsec:proof_Td}.

\subsection{Non-monotonicity of \texorpdfstring{$p_c$}{pc} as a function of the graph}
With the established phase transitions in certain graphs in Theorem \ref{thm:crit_genA_Z} and Theorem \ref{thm:crit_genA_T}, a natural question is what influence the underlying graph has on this phenomenon. In particular, how increasing a graph by adding additional vertices and edges affects the critical value $p_c$. In case of site percolation it is immediately clear that $\widetilde{p}_c(G_1) \ge \widetilde{p}_c(G_2)$ for any graphs $G_1,G_2$ such that $G_2$ contains $G_1$ as a subgraph, because an infinite path of open vertices in $G_1$ is also an infinite path of open vertices in $G_2$. However, for the frog model with death, it was shown by \cite{F2004} that increasing the graph on which the system evolves can in some cases decrease the critical parameter and in other cases increase it. The reason for this is that the graph needs to be explored by random walks to create new offspring, and thus, attaching large complete graphs to each vertex causes the random walks to get stuck inside these finite subgraphs and die before they can spread out further. A similar approach also applies for the SIMI, and in section \ref{subsec:proof_nonmoncrit} we will show the following result.
\begin{theorem}\label{thm:nonmon_graph}
There exist graphs $G,G_1,G_2$ such that $G$ is a subgraph of both $G_1$ and $G_2$, but the critical parameters for the SIMI with deterministically $3$ offspring at an infection satisfy
 \[
 1 > p_c(G_1,3) > p_c(G,3) > p_c(G_2,3) > 0.
 \]
\end{theorem}
\subsection{Recurrence to the origin}
In this section, for a quasi-vertex-transitive graph $G$, i.e., when the automorphism group Aut$(G)$ has finitely many orbits, we investigate recurrence to the origin. Let
\[
\vartheta(G,p,A) := \mathbb{P}\begin{pmatrix}\text{In the SIMI on }G \text{ with initially only the origin }\mathbf{0}\text{ infected,}\\\text{offspring distributed as }A \text{ and}\\\text{susceptible hosts appearing with probability } p,\\\text{the origin }\mathbf{0}\text{ is visited by a parasite infinitely often.}\end{pmatrix}.
\]
Our main result reads that on quasi-vertex-transitive graphs, this probability is $0$ for any $p<1$ and offspring distribution with finite expectation, i.e., the SIMI is a.s.~transient.
\begin{theorem}\label{thm:norecphase}
  Let $G$ be a quasi-vertex-transitive graph, $p\in[0,1)$, and $A$ be such that $\mathbb{E}[A]<\infty$. Then $\vartheta(G,p,A) = 0$.
\end{theorem}
\begin{remark}
  For the case $p=1$, the recurrence of this model on general graphs is still open. It is recurrent on any $\Z^d$ for $d\ge 1$ and any offspring distribution (\cite{P2001}). Also, depending on the offspring distribution, the frog model can be recurrent or transient on $\mathbb{T}_d$ for any $d\ge3$ (cf. \cite{J2016}). The frog model with deterministically $2$ offspring at an infection is recurrent on $\mathbb{T}_3$ and transient on $\mathbb{T}_6$, but the behaviour on $\mathbb{T}_4$ and $\mathbb{T}_5$ is not solved (cf. \cite{J2017}).
\end{remark}
\section{Construction of the Process}
In this section we construct the SIMI on $G$. Although it is intuitively clear that the survival probability should be monotone in the parameter $p$, Example \ref{ex:notmon} shows that we cannot conclude that simply by relying on the monotonicity of the classical frog model construction. For that reason, we consider two ways of constructing the model. The first approach is the classical way to assign each parasite that enters the system a label $(x,i)\in V\times\N$ and sample its entire path $(Y^{x,i}_n)_{n\ge0}$. We note that the path $(Y^{x,i}_n)_{n\ge0}$ will get sampled for all $n\ge0$ but only be used until the time that the parasite dies. The second approach will assign each vertex $x\in V$ sequences $(D^x_n)_{n\in\N}$ of jump directions that a parasite will perform to leave that vertex after jumping onto it. The two constructions satisfy two different important almost sure path properties that will be shown in Lemma \ref{lem:entpath_couple} and Lemma \ref{Lem:survdom}.

A classic tool from percolation theory is to couple the types of hosts for all $p\in[0,1]$ such that increasing $p$ almost surely only yields more susceptible hosts; see e.g., \cite[p. 11]{G1999}. To this end, we assume the following underlying probability space.
\begin{definition}\label{def:prob_space}
Let $G = (V,E)$ be a graph, and assume there is a probability space $\mathbf{\Omega}$ on which are defined the following independent random variables:
\begin{itemize}
    \item An i.i.d.~collection $\mathbf{U} := \{U_x:x\in V\}$ which are uniformly distributed on $(0,1)$ and for $p\in[0,1]$ we define $\mathbf{I}^p := \{I_x^p :x\in V\} := \{1+\infty\cdot\mathds{1}_{U_x > p}:x\in V\}$. A host at $x$ will be susceptible if $I_x^p = 1$ and immune if $I_x^p = \infty$.
    \item A $\N_0$-valued random variable $A$ with $\mathbb{E}[A] < \infty$ and an i.i.d.~collection $\mathbf{A}:=\{A_x:x\in V\}$ distributed as $A$. After a host at $x$ is infected, $A_x$ many parasites will be generated.
    \item An independent collection $\mathbf{Y} := \{(Y_n^{x,i})_{n\ge0}:x\in V,i\in\N\}$ of simple symmetric random walks in discrete time such that $Y_0^{x,i} = x$ for all $x\in V,i\in \N$.
    \item An independent collection $\mathbf{D} := \{D_n^x:x\in V,n\in\N\}$ such that $D_n^x$ is uniformly distributed over $\mathcal{N}_x\subset V$, the neighbourhood of $x$ in $G$, for all $x\in V,n\in\N$.
\end{itemize}
\end{definition}
\begin{remark}
 By definition of the coupling of $\{\mathbf{I}^p:p\in[0,1]\}$, if $0\le p_1\le p_2\le 1$ then $I_x^{p_1} \ge I_x^{p_2}$ for all $x\in V$. Thus, increasing the value of $p$ generates possibly more susceptible hosts in $\mathbf{I}^p$ while keeping all hosts susceptible that were already of that type for the smaller value of $p$.
\end{remark}
Next, we introduce some additional notation used throughout this paper. For a graph $G = (V,E)$, we recall the graph distance $d_G$ as the smallest number of edges in a path between two vertices.
For any finite set $\mathcal{J}\subset V$ we define its outer boundary $\partial\mathcal{J} := \{y\in V\setminus\mathcal{J}\vert\,\exists x\in\mathcal{J}: \{x,y\}\in E\}$ and set $\overline{\mathcal{J}} := \mathcal{J}\cup\partial\mathcal{J}$. Also, we recall the supremum-norm $\lVert x\rVert_\infty := \max\{\vert x_1\vert,\dots,\vert x_d\vert\}$ on $\mathbb{R}^d$ and define for $r\ge 0$ the closed cube $B_\infty^d(x,r) := \{y\in\R^d:\lVert x-y\rVert_\infty\le r\}$. On $\mathbb{T}_d$ we note that for any $x,y\in\mathbb{T}_d$ there is a unique path from $x$ to $y$ and say that $x\ge y$ if $y$ is contained in the unique path from $\mathbf{0}$ to $x$.

The state space of the process on some graph $G= (V,E)$ will be given by tupels
\[
\mathbb{S}:=\left\{(\mathcal{I},\eta)\left\vert\begin{matrix}\mathcal{I}\subset V \text{ finite}, \eta:\overline{\mathcal{I}}\to\N_0\cup\{-1,-\infty\}\text{ with }\\\eta(\mathcal{I})\subset \N_0,\eta(\partial\mathcal{I}) \subset\{-1,-\infty\}\end{matrix} \right.\right\}.
\]
The set $\mathcal{I}$ are the currently infected sites without hosts. $\eta$ gives the amount of parasites on each infected site $x\in\mathcal{I}$ and the immunity of each site $x\in\partial\mathcal{I}$ that can be reached in one jump from an infected site. For any $p\in(0,1]$ and initial configuration $(\mathcal{I},\eta)\in\mathbb{S}$, we will construct the SIMI $(\mathcal{I}^p_t,\eta_t^p)_{n\ge0}$ as a strong Markov process defined on $\mathbf{\Omega}$ and taking values in $\mathbb{S}$, equipped with the discrete topology, as well as prove basic properties of the so-coupled processes for different initial configurations in the upcoming subsections \ref{subsec:entpath} and \ref{subsec:vertjumps}.
\begin{definition}
Let $p\in(0,1]$ and $\mathcal{I}\subset V$ be finite. The random initial configuration $(\mathcal{I},\eta^{\text{reg},p})$ is defined as
  \[
  \eta^{\text{reg},p}(x) =\begin{cases} A_x, &x\in\mathcal{I}\\
   -I^p_x, &x\in\partial\mathcal{I}\end{cases}.
  \]
\end{definition}
\subsection{Parasite-wise path construction}\label{subsec:entpath}
In this subsection we construct the SIMI by assigning each parasite a label and sampling its entire path at once. We call this the \emph{parasite-wise} construction.

Suppose there is some initial configuration $(\mathcal{I},\eta)\in \mathbb{S}$ and a $p\in(0,1]$. To distinguish the two constructions, we will denote the process constructed in this section by $(\widetilde{\mathcal{I}^p_t}(\mathcal{I},\eta),\widetilde{\eta_t^p}(\mathcal{I},\eta))_{n\ge0}$. Also, we will not indicate the underlying graph $G= (V,E)$ on which the process evolves in our notation and always assume that the underlying graph is clear from the context.

Using the variables from Definition \ref{def:prob_space}, the system evolves as follows. We assign each parasite a label $(x,i)\in V\times\N$, where $x$ is the location of the parasite and $i$ enumerates the parasites on the same site $x$. A parasite with label $(x,i)$ moves along the trajectory $(Y^{x,i}_n)_{n\ge0}$ until it visits a vertex $y\in V$ that is not yet infected. If $I_y^p = \infty$, the parasite is killed and the label $(x,i)$ is removed from the system, and we no longer use $(Y^{x,i}_n)_{n\ge0}$. If $I_y^p= 1$ and $A_y = 0$, then the parasite is also killed and the label $(x,i)$ is removed from the system, but the site $y$ is added to the set of infected sites. If $I_y^p= 1$ and $A_y \ge 1$, then the site $y$ is added to the set of infected sites, the label $(x,i)$ remains in the system, and the labels $(y,1),\dots,(y,A_y-1)$ are added to the system as parasites that follow the trajectory $(Y^{y,j}_n)_{n\ge0}$ for $j=1,\dots, A_y-1$.
Also, we assume that there is some deterministic rule to determine the order in which the labels are processed in each time step.

In the following we make the assumption that $A\ge 1$ almost surely. Hence, the second case in the construction does not appear, and labels of parasites only get removed if the parasite visits an immune host. Then, by construction, the $i$-th parasite that is generated at some site $x\in V$ follows the path $(Y^{x,i}_n)_{n\ge0}$ until it visits a vertex $y\in V\setminus\mathcal{I}$ with $I^p_y = \infty$. Hence, it is useful to define for any $(x,i)\in V\times\N$ the lifetime and set of vertices visited by the parasite with label $(x,i)$:
\begin{equation}\label{eq:perc_coupl}
\begin{split}
\tau^p_{x,i}(\mathcal{I}) :&= \inf\{n\ge 0: Y^{x,i}_n \notin\mathcal{I}~\text{ and }~ I^p_{Y^{x,i}_n} = \infty\} \\
\mathcal{V}^p_{x,i}(\mathcal{I}) :&= \{Y^{x,i}_n: 0\le n < \tau^p_{x,i}(\mathcal{I})\}.
\end{split}
\end{equation}
From the construction it is clear that the set
\[\mathcal{I}_\infty^p(\mathcal{I},\eta^{\text{reg},p}) := \bigcup_{n\ge0}\widetilde{\mathcal{I}^p_n}(\mathcal{I},\eta^{\text{reg},p})\] of vertices that eventually get infected when the initial configuration is given by $(\mathcal{I},\eta^{\text{reg},p})$ can be described as exactly those vertices $x$ for which there exists an infection path from the initially infected sites to $x$, i.e., a sequence of vertices connecting $x$ to the initially infected vertices, such that a parasite from each vertex in the sequence reaches (and infects) the next vertex in the sequence before it dies. Precisely, a vertex $x\in V$ is in $\mathcal{I}_\infty^p(\mathcal{I},\eta^{\text{reg},p})$, if and only if $x\in\mathcal{I}$ or $x\in V\setminus\mathcal{I}$ and for some $n\ge 1$ there exist
\begin{equation}\label{eq:infection_sequence}x_0\in\mathcal{I},i_0\in\{1,\dots,A_{x_0}\}\text{ and }x_1,\dots,x_n\in V\setminus\mathcal{I},~i_1,\dots,i_{n-1}\in\N
\end{equation} with $1\le i_{m} \le A_{x_{m}}-1$ for all $m\in\{1,\dots n-1\}$, $x_n = x$ and $x_{k+1}\in\mathcal{V}^p_{x_k,i_k}(\mathcal{I})$ for all $k\in\{0,\dots, n-1\}$.

The following lemma is evident, using this representation of the eventually infected sites.
\begin{lemma}\label{lem:entpath_couple}
Let $0<p<p^\prime \le 1$ and $\mathcal{I}\subset\mathcal{I}^\prime\subset V$ be finite, and suppose that $A\ge 1$ almost surely. Then we have
  \[
  \mathcal{I}^p_\infty(\mathcal{I},\eta^{\text{reg},p}) \subset  \mathcal{I}^{p^\prime}_\infty(\mathcal{I}^\prime,\eta^{\text{reg},p^{\prime}})
  \]
  almost surely. Also, the event that the parasite population with initial configuration $(\mathcal{I},\eta^{\text{reg},p})$ survives for infinite time is given by the event that
  \[
  \{\vert\mathcal{I}^p_\infty(\mathcal{I},\eta^{\text{reg},p})\vert = \infty\}.
  \]
\end{lemma}
\begin{proof}
  Because $I^p_x \ge I^{p^\prime}_x$ for all $x\in V$, we obtain $\tau^p_{x,i}(\mathcal{I}) \le \tau^{p^\prime}_{x,i}(\mathcal{I}^\prime)$ and thus we also have the inclusion $\mathcal{V}^p_{x,i}(\mathcal{I}) \subset \mathcal{V}^{p^\prime}_{x,i}(\mathcal{I}^\prime)$ for all $(x,i)\in V\times\N$. This already shows the first claim, because for $x\in \mathcal{I}^p_\infty(\mathcal{I},\eta^{\text{reg},p})\setminus\mathcal{I}^\prime$ any sequence, as in \eqref{eq:infection_sequence}, contains a subsequence that satisfies \eqref{eq:infection_sequence} for $\mathcal{I}^{p^\prime}_\infty(\mathcal{I}^\prime,\eta^{\text{reg},p^{\prime}})$.

  The second claim is true, because $\tau^p_{x,i}(\mathcal{I})<\infty,\vert\mathcal{V}^p_{x,i}(\mathcal{I})\vert < \infty$ almost surely for all $(x,i)\in V\times \N$, which shows the $\subset$ inclusion, and it takes at least $d_G(x,y)$ time steps for the infection to spread from $x$ to $y$, which shows the $\supset$ inclusion.
\end{proof}
Finally, we give a concrete example that shows the coupling in this section is not almost surely monotone in the parameter $p$, if $\mathbb{P}(A=0) > 0$.
\begin{example}\label{ex:notmon}
We give a realisation of $\mathbf{Y},\mathbf{U},\mathbf{A}$ on $\Z^2$ where the process with a higher probability of susceptible hosts dies out and the process with lower probability survives in figure \ref{fig:nonmon}.

\begin{figure}[ht]\label{fig:nonmon}
\begin{center}\includegraphics[scale = 0.8]{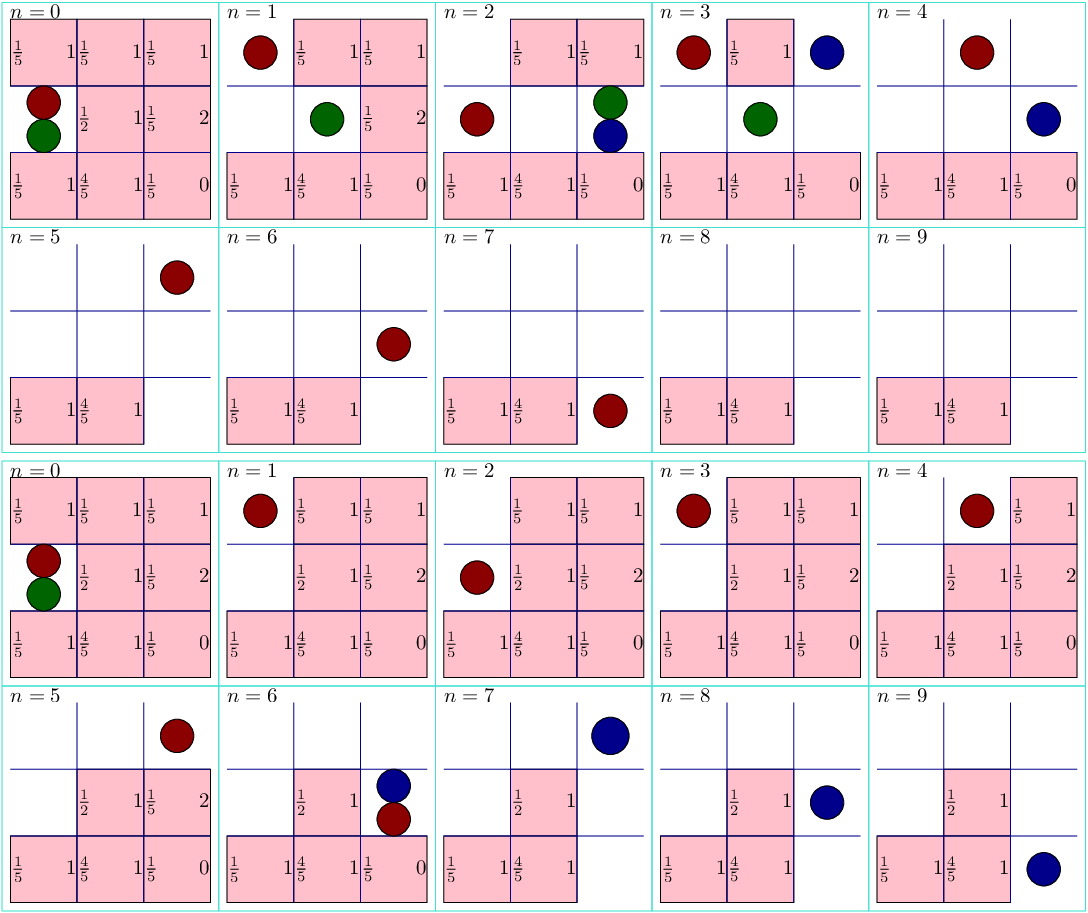}
\end{center}
 \caption{Each coloured square represents a living host, and each circle represents a parasite, with the colour corresponding to their label. The right number in each host is the value $A_x$ of offspring, and the left number is the value of $U_x$, determining if the host is immune, according to $U_x > \frac34$ in the upper part and $U_x > \frac14$ in the lower part.
}
\end{figure}
 We follow the infections for the cases $p^\prime = \frac{3}{4}$ in the upper part and for $p=\frac14$ in the lower part, starting from the same initial configuration and using the same random walks to determine the paths of the coloured parasites.
The different behaviour stems from the fact that the host in the centre is immune for $p=\frac{1}{4}$ and susceptible for $p=\frac{3}{4}$. Thus, it prevents the green parasite from waking up the blue parasite early. Thereby the blue parasite's death at the bottom left host is prevented by the sacrifice of the red parasite, which dies at the middle bottom host anyway. The blue parasite may then move to the right and start a new infection.
\end{example}
\subsection{Vertex-wise path construction}\label{subsec:vertjumps}
In this subsection we construct the SIMI by assigning each vertex a sequence of directions in which arriving parasites will leave the vertex. We call this the \emph{vertex-wise} construction.

The goal is to establish an analogue for Lemma \ref{lem:entpath_couple} in the case that $A=0$ with positive probability, which for the previous parasite-wise construction is not true, as seen in Example \ref{ex:notmon}. Suppose there is some initial configuration $(\mathcal{I},\eta)\in \mathbb{S}$ and a $p\in(0,1]$. For readability, we drop the reference to the initial configuration and in the following just write $\mathcal{I}^p_n,\eta^p_n$ instead of $\mathcal{I}^p_n(\mathcal{I},\eta),\eta^p_n(\mathcal{I},\eta)$. Also, we again leave out the reference to the graph on which the process is defined, as the construction steps are analogous for any graph. To perform the proof of Lemma \ref{Lem:survdom}, we need to introduce some additional notation. At each time step $n\ge0$, using the collection $\mathbf{D}$ from Definition \ref{def:prob_space}, we draw the jump directions that each parasite will perform and then calculate what the new state is after these jumps, according to the dynamics of the model. For any $x\in V$ we keep track of the number $N_n^p(x)$ of jumps that occurred on site $x$ prior to time step $n$, where initially $N_0^p(x) = 0$ for all $x\in V$. Then, we use the variables
\[
D_{N_n^p(x)+1}^x,\dots, D_{N_n^p(x) + \eta_n^p(x)}^x
\]
to determine the jump directions of the $\eta_n^p(x)$ many parasites sitting on site $x$ at time $n$, and the variables $\mathbf{A},\mathbf{I}^p$ and the set $\mathcal{I}^p_n$ to determine the outcome of these jumps and calculate the new configuration $(\mathcal{I}_{n+1}^p,\eta_{n+1}^p)$. In particular, we assume some deterministic labelling rule for the order of jumps and then set $N_{n+1}^p(x) := N_n^p(x) + \eta_n^p(x)$.
It is clear that the processes $(\mathcal{I}_n^p(\mathcal{I},\eta),\eta_n^p(\mathcal{I},\eta))_{n\ge0}$ and $(\widetilde{\mathcal{I}^p_n}(\mathcal{I},\eta),\widetilde{\eta^p_n}(\mathcal{I},\eta))_{n\ge0}$ have the same distribution, and we now want to show an equivalent of Lemma \ref{lem:entpath_couple} that also holds if $A = 0$ with positive probability.
\begin{lemma}\label{Lem:survdom}
Let $0<p < p^\prime\le 1$, $\emptyset\neq\mathcal{I}\subset\mathcal{I}^\prime\subset \Z^d$, and suppose $A$ takes values in $\N_0$. Then we have
  \[
  \mathcal{I}^p_n(\mathcal{I},\eta^{\text{reg},p}) \subset \mathcal{I}^{p^\prime}_n(\mathcal{I}^\prime,\eta^{\text{reg},p^\prime}) ~\text{almost surely for all }n\ge0.
  \]
\end{lemma}
\begin{proof}
For readability we drop the reference to the initial configuration and, e.g., simply write $\mathcal{I}^p_n$ instead of $\mathcal{I}^p_n(\mathcal{I},\eta^{\text{reg},p})$ and $\mathcal{I}_n^{p^\prime}$ instead of $\mathcal{I}_n^{p^\prime}(\mathcal{I}^\prime,\eta^{\text{reg},p^\prime})$. We do the same for any other variable in the construction of the process, i.e., for any variable used in the construction of $((\mathcal{I}_n^{p^\prime},\eta_n^{p^\prime})(\mathcal{I}^\prime,\eta^{\text{reg},p^\prime}))_{n\ge0}$ we only keep the sup-index $p^\prime$ to distinguish between the two processes.

\noindent For $n\ge 1$ we define the sets $\mathcal{S}_n^p := \{(x,k):x\in\mathcal{I}_n^p, 1\le k \le N_n^p(x)\}$ and $\mathcal{S}_n^{p^\prime}$ analogously. These sets contain the jumps that happened before time $n$ in each process. In particular, only the variables $D_k^x$ with $(x,k)$ in one of these sets are used before time $n$. Also, we define the sets of upcoming jumps $\mathcal{T}_n^p := \{(x,k): x\in\mathcal{I}_n^p,N_n^p(x) < k \le N_n^p(x) + \eta_n^p(x)\}$ and $\mathcal{T}_n^{p^\prime}$ analogously. We note that $\mathcal{S}_n^{p}\cup\mathcal{T}_n^p = \mathcal{S}_{n+1}^p$ and analogously for the primed variables.

We will show by induction that for all $n\ge0$ we have \begin{equation}\label{eq:IH}\mathcal{I}_n^p\subset\mathcal{I}_n^{p^\prime},N_n^p(x)\le N_n^{p^\prime}(x) \text{ for all }x\in\mathcal{I}_n^p\text{ and }\mathcal{T}_n^p \subset \mathcal{S}_{n+1}^{p^\prime}.\end{equation}
By assumption the induction hypothesis \eqref{eq:IH} holds for $n=0$. Suppose now that \eqref{eq:IH} holds for all $0\le m\le n$ and that $(x,k)\in\mathcal{T}_n^p$. By \eqref{eq:IH}, this jump is also performed in the primed process before time $n+1$, because $\mathcal{T}_n^p \subset\mathcal{S}_{n+1}^{p^\prime}$. But since $p < p^\prime$, any susceptible host in the unprimed process was also susceptible (or already infected at time $0$) in the primed process. In particular, the outcome of the jump $(x,k)$ leaves the conditions \eqref{eq:IH} in effect. To see this, observe that if $(x,k)$ produces no infection in the unprimed process, then \eqref{eq:IH} is trivially satisfied, and if $(x,k)$ produces an infection in the unprimed process, then this infection also happens (or has already happened) in the primed process.
  \end{proof}
This result may seem like the construction in Section \ref{subsec:entpath} was not necessary. However, in the case of $A\ge 1$, the fact that the infection can be calculated by following the sets $\{\mathcal{V}^p_{x,i}(\mathcal{I}):(x,i)\in V\times\N\}$ is an important feature that the vertex-wise construction in this section does not provide. The key consequence of this is that in the parasite-wise construction, the infection that started from some vertex $x$ almost surely is contained in the infection starting from some vertex $y\neq x$ after it infected $x$. In the vertex-wise construction of this section, this almost sure relation is no longer true, because if the process was already running when $x$ was infected, we use different parts of $\mathbf{D}$ to determine the jumps than we use if we freshly start the infection at $x$. In particular, for the proof of Theorem \ref{thm:crit_nontriv}, we make use of this feature of the parasite-wise construction to couple the process with a supercritical site percolation.
\section{Proofs of the main results}\label{sec:proofs}
In this section we prove the theorems stated in Section \ref{sec:mainres}. First we formally define the object of interest, the critical parameter $p_c(G,A)$, as follows.
\begin{definition}\label{def:crit_par}
Let $G=(V,E)$ be a graph. The value
\[
    p_c(G,A):= \inf\left\{p\in(0,1]:\mathbb{P}\left(\bigcap_{n\ge0}\{\# \eta_n^p(\{\mathbf{0}\}\subset V,\eta^{\text{reg},p}) > 0\}\right)>0\right\}
\]
    is called the critical parameter for the SIMI on $G$, where
    \[
    \#\eta_n^p(\{\mathbf{0}\}\subset V,\eta^{\text{reg},p}) := \sum_{x\in \mathcal{I}_n^p(\{\mathbf{0}\}\subset V,\eta^{\text{reg},p})} \eta_n^p(\{\mathbf{0}\}\subset V,\eta^{\text{reg},p})(x)
    \]
    is the amount of living parasites at time $n$.
    \end{definition}
    \begin{remark}
Note that by Lemma \ref{Lem:survdom}, the survival probability is monotone in $p\in(0,1]$. Thus, $p_c(G,A)$ is actually the critical parameter.
    \end{remark}

    \subsection{Proofs for survival results on \texorpdfstring{$\Z^d$}{Zd}}\label{subsec:proof_Zd}
    \begin{proof}[Proof of Theorem \ref{thm:crit_nontriv}]
    For $p$ close enough to $1$, we will couple the SIMI with a supercritical site percolation on $\Z^d$ such that if there is percolation, then the parasite population in the SIMI survives. The idea is to find large areas with only susceptible hosts that will be completely infected if one of the hosts in the centre of that area gets infected. Finding an infinite chain of these areas will yield the result.

    We note that for $p=1$, and because $A\ge 1$ almost surely, the construction in Section \ref{subsec:entpath} is exactly the construction of the classical frog model with random offspring distribution given by $A-1$. For this model, the set of infected sites satisfies a shape theorem (cf. \cite[Theorem 1.1]{A2001}). For $x\in\Z^d$ we abbreviate
    \[
    \xi^x_n := \widetilde{\mathcal{I}^1_n}(\{x\},\eta^{\text{reg},1}-\delta_x)
    \]
    for the set of infected sites in the SIMI at time $n$ with initially only site $x$ infected and $p =1$. We note here that we initially place only $A_x-1$ parasites on site $x$ because only those will be added in the system if site $x$ is infected in an already running process. Also, we note that due to the assumption $\mathbb{E}[A] > 1$, we have $\mathbb{P}(A \ge 2)>0$.

    \noindent Next, we set \[\overline{\xi}_n^x := \left\{y+\left(-\frac12,\frac12\right]^d:y\in\xi_n^x\right\}.\]
Let $\emptyset\neq\mathcal{A}_d\subset \R^d$ be the asymptotic shape of the frog model with offspring distributed as $A-1$. This means that for any $\varepsilon \in(0,1)$ there is an almost surely finite $n_0 = n_0(\varepsilon)$ such that
\[
(1-\varepsilon)\mathcal{A}_d\subset \frac{\overline{\xi}_n^\mathbf{0}}{n}\subset (1+\varepsilon)\mathcal{A}_d,~\text{ for all } n\ge n_0.
\]
We fix some $\varepsilon \in(0,1)$, then choose $r>0$ such that $B_\infty^d(\mathbf{0},r) \subset \mathcal{A}_d$, which is possible since $\mathcal{A}_d$ is symmetric, convex, and contains more than $1$ point (see \cite[Theorem 1.1]{A2001}), and set
\[
R(N) := \left\lceil\frac{N}{r(1-\varepsilon)}\right\rceil.
\]
By taking $N\in\N$ large enough, we can make
\[
\theta(N) := \mathbb{P}\left(\left.(1-\varepsilon)B_\infty^d(\mathbf{0},r)\subset \frac{\overline{\xi}^\mathbf{0}_{R(N)}}{R(N)}\right\vert A_0 \ge 2\right)
\]
as close to $1$ as we want. By the definition of $R(N)$, the event $(1-\varepsilon)B_\infty^d(\mathbf{0},r)\subset \frac{\overline{\xi}^\mathbf{0}_{R(N)}}{R(N)}$ implies that $B_\infty^d(\mathbf{0},N)\subset \overline{\xi}^\mathbf{0}_{R(N)}$. This means that with a probability $\theta(N)$ as close to $1$ as we want, the entire cube $B_{\infty}^d(\mathbf{0},N)$ is infected after $R(N)$ steps.

\noindent For the coupling of the SIMI with $p\neq 1$ with a supercritical site percolation, we will make $p$ close enough to $1$ (depending on $N$) such that with sufficiently high probability, for each $x\in\Z^d$ there is a large area $V_N(x)$ around $B_\infty^d(x,N)$ with only susceptible hosts. Thus, if $x$ is infected, then after $R(N)$ steps, the entire ball $B_\infty^d(x,N)$ will be infected with probability at least $\theta(N)$, because the SIMI inside $V_N(x)$ is just the frog model. In particular, the existence of an infinite chain of such balls $(B_\infty^d(x_n,N))_{n\ge1}$ with $x_{n+1} \in B_\infty^d(x_n,N)$ implies the survival of the parasite population. For convenience we assume that $N$ is odd and for
\[
x\in N\Z^d = \{(x_1,\dots,x_d)\in\Z^d: N\vert x_1,\dots,N\vert x_d\}
\]
define $V_N(x) := B_\infty^d\left(x,\frac{N-1}{2}+R(N)\right)$.
We now declare a site percolation on $N\Z^d$ as follows, which we will also sketch in Figure \ref{fig:perc}.

 \begin{figure}[ht]
  \begin{center}\includegraphics[scale = 0.8]{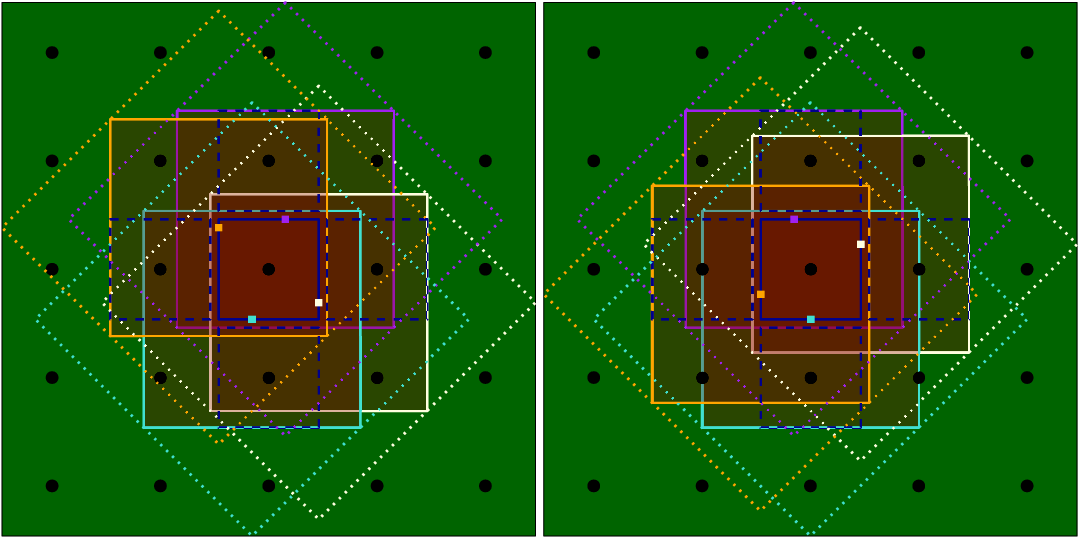}\end{center}
  \caption{Open sites in the constructed percolation on $N\Z^2$.
  The black dots represent the vertices of $N\Z^2$; the remaining vertices of $\Z^2$ are omitted for clarity.
  The vertex $x$, located at the centre of the left green square, is open, meaning the following. All hosts within the green square $V_N(x)$ are susceptible.
  Furthermore, on each face of the blue square of side length $N-1$ centred at $x$, there is a coloured vertex such that an infection starting at this vertex spreads throughout the dashed blue square surrounding the neighbouring vertex in $N\Z^2$ in the corresponding direction.
  More precisely, an infection initiated at a coloured vertex infects any host inside the pale red region within the square of side length $2N$ of the corresponding colour around that vertex within $R(N)$ time steps.
  During this time period, the infection does not reach (and therefore does not depend on) any vertices outside the dotted squares of the same colour.
  Consequently, the event that $x$ is open depends only on the hosts and parasites within the green square $V_N(x)$.
  Hence, this event depends on the open/closed status of only a deterministic number of other vertices in $N\Z^2$.
  By construction, an open vertex with its entire blue square infected also transmits the infection to the blue squares of all its neighbouring vertices in $N\Z^2$, and an infinite sequence of open vertices sustains the infection indefinitely.}
  \label{fig:perc}
 \end{figure}
\noindent We say $x\in N\Z^d$ is open if
\[
G_I(x):= \bigcap_{y\in V_N(x)\cap\Z^d} \{I_y^p = 1\} ~\text{ and }~G_F(x)  := \bigcap_{L\in F(x)}\bigcup_{j\in L\cap\Z^d} \{A_{j} \ge 2, B_\infty^d(j,N)\subset\overline{\xi}^{j}_{R(N)}\}
\]
where $F(x)$ are $2d$ many $(d-1)$-dimensional faces of the cube $B^d_\infty\left(x,\frac{N-1}{2}\right)$.
The event $G_I(x)\cap G_F(x)$ means that for each face $L\in F(x)$ there is a vertex $j \in L$ such that if $j$ gets infected, then in the frog model and, due to the event $G_I(x)$, also in the SIMI a sufficiently large ball around $j$ will be completely infected after $R(N)$ time steps. Clearly $G_I(x)$ and $G_F(x)$ are independent for each $x\in N\Z^d$. Also, by the independence of $\mathbf{I}^p$, we have $\mathbb{P}(G_I(x)) = p^{(N+2R(N))^d}$, and to calculate the probability of $G_F(x)$ we first note that for each face $L\in F(x)$
\[
\mathbb{P}\left(\bigcup_{j\in L\cap\Z^d}\{A_j \ge 2, B_\infty^d(j,N)\subset\overline{\xi}^{j}_{R(N)}\}\right) \ge (1-\mathbb{P}(A = 1)^{N^{d-1}})\theta(N).
\]
Trivially, the $2d$ many events in the intersection over $L\in F(x)$ are $2d$-dependent, i.e., the event for one $L\in F(x)$ only depends on at most $2d$ many events for other $L^\prime\in F(x)$. Thus, applying \cite[Theorem B26]{L1999}, we have
\[
\mathbb{P}(G_F(x)) \ge \left(1-\left(1-\theta(N)\left(1-\mathbb{P}(A = 1)^{N^{d-1}}\right)\right)^{\frac{1}{2d}}\right)^{4d}.
\]
Taking $N$ large enough and $p$ close to $1$, we can make $\mathbb{P}(G_I(x),G_F(x)) \ge \rho$ for any $\rho \in (0,1)$.
Next we observe that for any $L\in F(x)$, the frog model started from any $j\in L$ cannot leave the set $V_N(x)$ until time $R(N)$, because parasites can only move to their neighbours in each time step. Hence, if $\lVert x-x^\prime\rVert_\infty \ge N+2R(N)$, then $G_I(x)\cap G_F(x)$ is independent of $G_I(x^\prime)\cap G_F(x^\prime)$. Counting the amount of $x^\prime \in N\mathbb{Z}^d$ with this distance to $x$, we thus obtain that the collection of variables $\{\mathds{1}_{G_I(x),G_F(x)} : x\in N\Z^d\}$ is $K$-dependent with $K = \left(6+\frac{4}{r(1-\varepsilon)}\right)^d$.
 In particular, choosing $\rho\in (0,1)$ such that $\left(1-(1-\rho)^{\frac{1}{K}}\right)^2 > \widetilde{p}_c(\Z^d)$, where $\widetilde{p}_c(\Z^d)$ is the critical parameter of Bernoulli site percolation on $\Z^d$, we obtain using \cite[Theorem B26]{L1999} again that with positive probability there is percolation.

Now, using Lemma \ref{lem:entpath_couple}, in the event that $\mathbf{0}$ is in the infinite cluster of this percolation and that the initial parasites in $\mathbf{0}$ reach all sites of $B_\infty^d\left(\mathbf{0},\frac{N+1}{2}\right)$ before leaving $V_N(\mathbf{0})$, the SIMI will survive for infinite time. As argued above, this happens with positive probability, and thus we can conclude that $p_c(\Z^d,A) < 1$. \qedhere
\end{proof}

\begin{proof}[Proof of Theorem \ref{thm:crit_genA_Z}]
Consider $\widetilde{I}^p_x := 1 + \infty\mathds{1}_{U_x > p} + \infty\mathds{1}_{A_x = 0}$. Using these variables as immunities, only vertices with $A_x \ge 1$ can be infected, and the resulting process will be given by the SIMI with $\widetilde{p} := \max\{0,p - \mathbb{P}(A=0)\}$ and offspring distributed as $\widetilde{A}$. Also, this process clearly is dominated by the original process, with offspring distributed as $A$ and susceptible hosts appearing with probability $p$. Hence the claim follows by Theorem \ref{thm:crit_nontriv}.
\end{proof}

\subsection{Proofs for survival results and asymptotic critical probability on \texorpdfstring{$\mathbb{T}_d$}{Td}}\label{subsec:proof_Td}
\begin{proof}[Proof of Theorem \ref{thm:Tdphase}]
  The proof is a minor adaptation of the proof for the corresponding result in \cite[Theorem 1.5]{A2002a}. For $x\neq \mathbf{0}\in\mathbb{T}_d$ we denote by $\mathbb{T}_d^+(x) = \{y\in\mathbb{T}_d:y\ge x\}$. Fix some vertex $x_{0,\mathbf{0}}$ adjacent to the root $\mathbf{0}$ and set $\mathbb{T}_d^+ := \mathbb{T}_d\setminus\mathbb{T}_d^+(x_{0,\mathbf{0}})$. For $(x,i)\in \mathbb{T}_d\times\N$ we denote by $\mathcal{W}_n^{x,i} := \{Y^{x,i}_m:0\le m\le n\}$ the set of vertices ever visited by the parasite $(x,i)$ up to time $n$ in the case $p=1$, and simply write $\mathcal{W}_n := \mathcal{W}_n^{\mathbf{0},1}$. Also, for any finite connected set $\mathcal{J}\subset\mathbb{T}_d$ with $\mathbf{0}\in\mathcal{J}$, we denote by
  \[
\partial_e(\mathcal{J}) := \mathcal{J}\setminus\{x\in \mathcal{J}\vert\exists y\in\mathcal{J}: y>x\}
  \]
  the external boundary of $\mathcal{J}$, that is, the set of vertices from which $\mathcal{J}$ can be left in one jump. We make the following observations:
  Since a random walk on $\mathbb{T}_d$ is transient, we have
  \begin{itemize}
    \item with positive probability the random walk never leaves $\mathbb{T}_d^+$, i.e., $\mathcal{W}_n\subset\mathbb{T}_d^+$ for all $n\ge 0$ with positive probability.
    \item $\vert \partial_e(\mathcal{W}_n)\vert$ is nondecreasing and $\vert\partial_e(\mathcal{W}_n)\vert\to\infty$ almost surely as $n\to\infty$.
  \end{itemize}
  To show a positive survival probability, we construct a supercritical Bienaymé-Galton-Watson branching process $(\xi_n)_{n\ge0}$ that is dominated by the number of living parasites in the SIMI. Initially we set $\zeta_0 = \{\mathbf{0}\}$ and $\xi_0 := \vert\zeta_0\vert$. If all parasites on $\mathbf{0}$ leave $\mathbb{T}_d^+$, i.e., for all $1\le i\le A_\mathbf{0}$ we have $\mathcal{W}^{\mathbf{0},i}_n\not\subset\mathbb{T}_d^+$
  for some $n\ge 0$, then we set $\zeta_1 := \emptyset,\xi_1 := 0$, and the Bienaymé-Galton-Watson branching process dies out. Else, we let
  \[
  j_{1,\mathbf{0}} := \inf\{1\le  i\le A_{\mathbf{0}}\vert\forall n\ge0: \mathcal{W}^{\mathbf{0},i}_n\subset\mathbb{T}_d^+\}
  \]
  be an index of a parasite that stays inside $\mathbb{T}_d^+$ and let
  \[
  \zeta_1 := \partial_e\left(\mathcal{V}_{\mathbf{0},j_{1,\mathbf{0}}}^p(\{\mathbf{0}\})\right)\setminus\left\{Y^{\mathbf{0},j_{1,\mathbf{0}}}_{\tau_{\mathbf{0},j_{1,\mathbf{0}}}^p-1}\right\}\subset\mathbb{T}_d^+
  \]
  be the external boundary of the set of sites visited by the parasite $(\mathbf{0},j_{1,\mathbf{0}})$, excluding the site from which it jumped onto a completely immune host. Again we set $\xi_1 := \vert\zeta_1\vert$. We note that by construction we have $\mathbb{T}_d(x) \cap\mathbb{T}_d(y) = \emptyset$ for $x,y\in\zeta_1$ with $x\neq y$.

  \noindent For $n\ge 1$, if $\zeta_n = \emptyset$, we set $\zeta_{n+1} = \emptyset$ and $\xi_{n+1} = 0$. Else, for each $x\in\zeta_n$, we determine the offspring of $x$ as follows. If $A_{x}= 1$, then $x$ has no offspring. If $A_{x}\ge2$ and for all $1\le i \le A_{x}-1$ we have $\mathcal{W}^{x,i}_n\not\subset\mathbb{T}_d(x)$ for some $n\ge 1$, then $x$ has no offspring. Else, we set
  \[
  j_{n+1,x} := \inf\{1\le i\le A_{x}-1\vert \forall n\ge 1:\mathcal{W}^{x,i}_n\subset\mathbb{T}_d(x)\}
  \]
  and define the offspring of $x$ as
  \[
  \partial_e\left(\mathcal{V}_{x,j_{n+1,x}}^p(\{x\})\right)\setminus\left\{Y^{x,j_{n+1,x}}_{\tau_{x,j_{n+1,x}}^p-1}\right\}\subset\mathbb{T}_d^+(x).
  \]
Finally we define $\zeta_{n+1}$ as the union of all offspring of elements in $\zeta_n$ and $\xi_{n+1} := \vert\zeta_{n+1}\vert$. Because the exploration for offspring only considers a so-far unexplored subtree, we obtain that $(\xi_n)_{n\ge0}$ is a Bienaymé-Galton-Watson branching process. Also, it is clear that by taking $p$ close to $1$ we can make it supercritical, because the number of visited sites $\vert\mathcal{V}_{x,j_{n+1,x}}^p(\{x\})\vert$ has a geometric distribution with parameter $1-p$ and thus goes to $\infty$ as $p$ approaches $1$.
\end{proof}
\begin{proof}[Proof of Theorem \ref{thm:crit_genA_T}] The proof of Theorem \ref{thm:crit_genA_T} is completely analogous to the proof of Theorem \ref{thm:crit_genA_Z}, substituting Theorem \ref{thm:Tdphase} for Theorem \ref{thm:crit_nontriv}. \end{proof}
\begin{proof}[Proof of Theorem \ref{thm:Tdasympt}]
The proof is analogous to the one for the corresponding result in \cite[Theorem 1.7]{A2002a}. To highlight why this approach also works in our model, we give the proof here again.

Since by Theorem \ref{thm:trivpos} we have $p_c(\mathbb{T}_d,A)\ge\frac{1}{\mathbb{E}[A]}$, it suffices to show that for any $p>\frac{1}{\mathbb{E}[A]}$ the SIMI survives on $\mathbb{T}_d$ for large enough $d$. For $s\in\N$, let $A^{(s)} := A\mathds{1}_{A\le s}$ and note that by monotone convergence $\mathbb{E}[A^{(s)}]\to\mathbb{E}[A]$. Thus, for $s$ large enough, $p>\frac{1}{\mathbb{E}[A^{(s)}]}$ and it suffices to show that the SIMI on $\mathbb{T}_d$ with offspring distributed as $A^{(s)}$ survives with positive probability. We construct an auxiliary process $(\widetilde{\xi}_n)_{n\ge0}$ that is dominated by the living parasites at time $n$ and also dominates a supercritical Bienaymé-Galton-Watson branching  process. Initially all parasites on $\mathbf{0}$ belong to $\widetilde{\xi}_0$. For $n\ge0$, the collection $\widetilde{\xi}_{n+1}$ will consist of parasites at distance $n+1$ to the root and will be constructed as follows. We assume some deterministic rule for the order in which the parasites are treated. Then if a parasite in $\widetilde{\xi}_n$ jumps, the following happens:
  \begin{itemize}
    \item If the parasite jumps towards the root, it is removed without offspring.
    \item If the parasite jumps onto some vertex at distance $n+1$ from the root and a different parasite has already jumped to that vertex, then it is removed without offspring.
    \item If the parasite jumps onto some vertex at distance $n+1$ from the root, it is the first parasite to jump onto that vertex, and the host on that vertex is completely immune, then it is removed without offspring.
    \item If the parasite jumps onto some vertex at distance $n+1$ from the root, is the first parasite to jump onto that vertex, and the host on that vertex is susceptible, then any new parasites generated (according to $A^{(s)})$ on that vertex are added to $\widetilde{\xi}_{n+1}$ as its offspring.
  \end{itemize}
  We note that by definition, each vertex can only be inhabited by at most $s$ parasites in $\widetilde{\xi}_n$. Also, because offspring can only get generated whenever a parasite jumps onto a never-before-visited vertex, the probability for that host to be susceptible is $p$, independent of everything else in $\widetilde{\xi}_n$, which is the crucial fact as to why the approach of \cite{A2002a} still works in our model. Hence, $(\widetilde{\xi}_n)_{n\ge0}$ dominates a Bienaymé-Galton-Watson branching process with mean offspring
 \[
 \frac{d-s}{d}\mathbb{E}[A^{(s)}]p.
 \]
 This corresponds to the worst case that $d-(s-1)$ directions are already used up by the other parasites on the vertex and $1$ direction is towards the root, resulting in at least $d-s$ free directions.
Taking $d$ large enough, this can be made greater than $1$, and thus the Bienaymé-Galton-Watson branching process, and hence the SIMI, survives with positive probability.
\end{proof}

\subsection{Proof for the non-monotonicity of \texorpdfstring{$p_c$}{pc} as a function of the graph}\label{subsec:proof_nonmoncrit}~

In this section we aim to show that the critical parameter for survival is not monotone in the underlying graph. To be precise, we will construct a sequence of graphs $(G^{(n)})_{n\ge1}$ such that each $G^{(n)}$ contains $\mathbb{T}_3$ as a subgraph, but the critical parameters with almost surely $3$ offspring at an infection satisfy $p_c(G^{(n)},3) \to 1$ as $n\to\infty$. The approach is similar to that of \cite{F2004} and relies on attaching a complete graph with some large degree $n$ to each vertex of $\mathbb{T}_3$, in which parasites get trapped and eventually killed.

Denote by $E_3$ the set of edges of the graph $\mathbb{T}_3$. Let $x\in\mathbb{T}_3$ and
\[G_n(x) := (V_n(x),E_n(x)) := (\{(x,0),(x,1),\dots,(x,n)\},\{\{(x,i),(x,j)\}:0\le i < j \le n\})
\]
be a complete graph with $n+1$ vertices attached to $x$. Consider the graph $G^{(n)} = (V^{(n)},E^{(n)})$ with
\[
V^{(n)} = \bigcup_{x\in\mathbb{T}_3}V_n(x),~~E^{(n)} = \{\{(x,0),(y,0)\}:\{x,y\}\in E_3\}\cup\bigcup_{x\in\mathbb{T}_3} E_n(x).
\]
\begin{proposition}\label{prop:compgraph}
    For any $p\in (0,1)$ there is an $n_0\in\N$ such that for all $n\ge n_0$ we have
    \[
    p_c(G^{(n)},3) \ge p.
    \]
    \end{proposition}
\begin{proof}
Let $n\in\N$ and $p\in(0,1)$. We will couple the SIMI on $G^{(n)}$ with a Bienaymé-Galton-Watson branching process, such that extinction of this Bienaymé-Galton-Watson branching process implies the extinction of the SIMI, and then show that for $n$ large enough the Bienaymé-Galton-Watson branching process is subcritical.

For readability we will (in some cases) drop the reference to $n$ and $p$, and because $A=3 \ge 1$ almost surely we consider the SIMI defined through the parasite-wise coupling. Initially, only the parasites in $(\mathbf{0},0)$ are placed, and thus we set \[\zeta_0:= \{((\mathbf{0},0),1)\},\] $\xi_0 = 1$, noting that the parasites $((\mathbf{0},0),2),((\mathbf{0},0),3)$ will in the next step be considered as potential offspring of $((\mathbf{0},0),1)$. Also, we let $U_0 = \emptyset\subset V^{(n)}$ be the vertices already explored in the construction of the Bienaymé-Galton-Watson branching process. For $m\ge 1$ we now inductively define the variables $\zeta_m,\xi_m = \vert\zeta_m\vert,U_m$ as follows. Set $U_m(0) := U_{m-1}$ and use some deterministic way to order $\zeta_{m-1}$ (e.g., lexicographically) as
\[
\zeta_{m-1} = \{((x_m(1),k_m(1)),i_m(1)),\dots,((x_{m}(\ell_m),k_{m}(\ell_m)),i_{m}(\ell_m)\}.
\]
For $j = 1,\dots, \ell_m$ we now inductively define $\zeta_m(j),U_m(j)$ as follows. We ignore the potential death of the parasite with label $Z_m (j) := ((x_m(j),k_m(j)),i_m(j))$ and let
\[
t_m(j) := \inf\{t\ge0: Y^{Z_m(j)}_t \notin U_m(j-1)\}
\]
be the time it leaves the explored subgraph $U_m(j-1)$ and visits a new vertex and let
\[
(y_m(j),0) := Y^{Z_m(j)}_{t_m(j)}
\]
be the vertex at which it does so (by construction it will only be possible to leave $U_m(j-1)$ at the ground level). The offspring of $Z_m(j)$ will be generated as follows. First, the host at $(y_m(j),0)$ gets infected and killed (even if it is immune), and we generate the parasites with labels $((y_m(j),0),1)$ and $((y_m(j),0),2)$ (respectively, for $m=1,j=1$ and thus $y_1(1) = \mathbf{0}$, we generate $((\mathbf{0},0),2),((\mathbf{0},0),3)$ and also replace the corresponding labels in the upcoming generation of offspring accordingly). Then we place the immune hosts inside $V_n(y_m(j))\setminus\{(y_m(j),0)\}$ according to $\mathbf{I}^p$ and denote by
\[
\mathcal{H}_m(j) := \{(y_m(j),k): 1\le k \le n \text{ and } I^p_{(y_m(j),k)}  =\infty\}
\]
the vertices with immune hosts. Next, we generate all parasites
\[
\{((y_m(j),1),1),(y_m(j),1),2),\dots,((y_m(j),n),1),(y_m(j),n),2)\}
\] inside $V_n(y_m(j))\setminus\{(y_m(j),0)\}$ and immediately kill off the parasites that are placed on top of an immune host. Finally, we let the remaining parasites (including $((y_m(j),0),1)$, $((y_m(j),0),2)$ and $Z_m(j)$) move according to $\mathbf{Y}$ and keep only those that leave $V_n(y_m(j))$ before they die at an immune host. Precisely, recalling the notation \eqref{eq:perc_coupl}, we set
\[
\zeta_m(j)^\prime := \left\{((y_m(j),k),i):\begin{matrix}1\le k \le n,~i\in\{1,2\},I^p_{(y_m(j),k)} = 1 \text{ and }\\ \mathcal{V}^p_{(y_m(j),k),i}(\{(y_m(j),k)\}) \not\subset V_n(y_m(j))\end{matrix} \right\},
\]
\[
\zeta_m(j)^{\prime\prime} := \left\{((y_m(j),0),i):i\in\{1,2\} \text{ and } \mathcal{V}^p_{(y_m(j),0),i}(\{(y_m(j),0)\}) \not\subset V_n(y_m(j))\right\}
\]
and
\[
\zeta_m(j)^{\prime\prime\prime} := \begin{cases}
                                  \{Z_m(j)\} ,& \exists t \ge t_m(j):\begin{matrix} Y^{Z_m(j)}_t \notin V_n(y_m(j)) \text{ and }\\ \forall t_m(j)\le k < t: Y^{Z_m(j)}_k\in V_n(y_m(j))\setminus\mathcal{H}_m(j) \end{matrix}\\
                                  \emptyset ,& \text{ else,}
                                 \end{cases}
\]
and then set $\zeta_m(j) = \zeta_m(j)^\prime\cup\zeta_m(j)^{\prime\prime}\cup\zeta_m(j)^{\prime\prime\prime}$. To conclude the $j$-th step, we set the explored subgraph $U_m(j) := U_m(j-1)\cup V_n(y_m(j))$ and continue with the next parasite in $\zeta_{m-1}$. Finally, we set $U_m := U_m(\ell_m)$,
\[
\zeta_m = \bigcup_{j=1}^{\ell_m} \zeta_m(j)
\]
and $\xi_m = \vert\zeta_m\vert$.  By construction and observation \eqref{eq:infection_sequence}, since the paths of all potentially living parasites are considered, we have
\[
\limsup_{m\to\infty}\xi_m \ge \limsup_{t\to\infty} \#\eta_t^p(\{(\mathbf{0},0)\}) ~~\text{almost surely.}
\]
Also, because the production of the offspring of $((x_m(j),k_m(j)),i_m(j))$ only depends on the immu\-nity of hosts inside $V_n(y_m(j))$, the paths of parasites generated inside $V_n(y_m(j))$ and the path of $((x_m(j),k_m(j)),i_m(j))$ after the time $t_m(j)$, we obtain that $\vert\zeta_m(j)\vert$ is independent of \[\xi_0,\dots,\xi_{m-1},\vert\zeta_m(1)\vert,\dots,\vert\zeta_m(j-1)\vert.\] In particular $(\xi_m)_{m\ge1}$ is a Bienaymé-Galton-Watson branching process. Thus, it remains to show that this Bienaymé-Galton-Watson branching process becomes subcritical for large $n$, i.e.,
\[
\mathbb{E}[\xi_1] \le 1
\]
for some large enough $n$ and fixed $p\in (0,1)$.

First, we let
\[
\mathcal{H} :=  \mathcal{H}_1(1) = \{(\mathbf{0},k)\in V^{(n)}: 1\le k \le n, I^p_{(\mathbf{0},k)} = \infty\}
\]
be the vertices with immune hosts in the complete graph attached to $(\mathbf{0},0)$ and, for $0\le k \le n,i\in\mathbb{N}$, define
\[
\sigma(k,i) := \inf\{m>0: Y^{(\mathbf{0},k),i}_m = (\mathbf{0},0)\},~~\tau(k,i) = \inf\{m\ge0:Y_m^{(\mathbf{0},k),i} \in\mathcal{H}\}
\]
as the times the parasite $((\mathbf{0},k),i)$ returns to $(\mathbf{0},0)$, respectively, when it is on top of an immune host inside $V_n(\mathbf{0})$. Clearly
\[\{\sigma(k,i):1\le k \le n,i\in\mathbb{N}\}~\text{ and }~\{\tau(k,i):1\le k \le n,i\in \mathbb{N}\}\] are identically distributed, respectively, and we simply denote $\sigma := \sigma(1,1),\tau:= \tau(1,1)$.

By independence of $\mathbf{I}^p$, we have that the number of immune hosts $L:=\vert\mathcal{H}\vert\sim \text{Bin}(n,1-p)$ and a parasite starting at some $(\mathbf{0},k)$ with $1\le k \le n$ has a probability of $\frac{n-L}{n(L+1)}$ to reach $(\mathbf{0},0)$ before getting killed by an immune host, i.e.,
\[
\mathbb{P}(\sigma < \tau \vert L) = \frac{n-L}{n}\cdot\frac{1}{L+1}.
\]
This probability consists of not being on top of an immune host, which happens with probability $\frac{n-L}{n}$, and then leaving the set of vertices with susceptible hosts through $(\mathbf{0},0)$ and not one of the $L$ immune hosts, which happens with probability $\frac{1}{L+1}$. If a parasite is at $(\mathbf{0},0)$, then with probability $\frac{3}{3+n}$, it leaves $V_n(\mathbf{0})$. Otherwise, the event that the parasite revisits $(\mathbf{0},0)$ before getting killed by an immune host is given by first jumping onto a vertex with a susceptible host and then leaving the set of vertices with susceptible hosts through the vertex $(\mathbf{0},0)$. In particular, conditionally on $L$, it has probability  $\frac{n-L}{n+3}\cdot\frac{1}{L+1}$. We denote by
\[
\gamma(k,i):= \inf\{m\ge0:Y^{(\mathbf{0},k),i}_m \notin V_n(\mathbf{0})\},
\]
 the time until parasite $((\mathbf{0},k),i)$ leaves $V_n(\mathbf{0})$ (ignoring its death) and again denote $\gamma = \gamma(1,1)$.
 Noting that, conditionally on $L$, each attempt to get back to $(\mathbf{0},0)$ is independent from the previous one, we obtain the probability to leave $V_n(\mathbf{0})$ before death is given by
\[
\begin{split}
\mathbb{P}(\gamma < \tau\vert L = \ell) &= \frac{3(n-\ell)}{(\ell+1)(3+n)n}\sum_{j=0}^\infty\left(\frac{n-\ell}{(n+3)(\ell+1)}\right)^j \\
&= \frac{3(n-\ell)}{(\ell+1)(3+n)n} \cdot\frac{(n+3)(\ell+1)}{(n+4)\ell+3} = \frac{3(n-\ell)}{n((n+4)\ell+3)}
\end{split}
\]
for $\ell= 0,\dots,n$. For $\ell=1,\dots,n$ we further estimate
\[
\frac{3(n-\ell)}{n((n+4)\ell+3)} \le \frac{3}{(n+4)\ell+3} \le \frac{1}{n+4}\cdot\frac{3}{\ell+\frac{3}{n+4}} \le \frac{1}{n+4}\cdot\frac{3}{\ell}.
\]
Now using that $L\sim \text{Bin}(n,1-p)$ we compute for large enough $n$ that
\[
\begin{split}
\mathbb{P}(\gamma < \tau) &= \sum_{\ell=0}^n\binom{n}{\ell}p^{n-\ell}(1-p)^\ell\frac{3(n-\ell)}{n((n+4)\ell+3)} \\
&\le p^{n} + \frac{3}{n+4}\sum_{\ell=1}^n\binom{n}{\ell}p^{n-\ell}(1-p)^\ell\frac{1}{\ell} \\
&\le p^{n} + \frac{3}{n+4}\cdot\frac{2}{n(1-p)},
\end{split}
\]
where we used \cite[Theorem 1]{M1999} with $k=1$ to estimate the inverse binomial moment
\[
\sum_{\ell=1}^n\binom{n}{\ell}p^{n-\ell}(1-p)^\ell\frac{1}{\ell} = \frac{1}{n(1-p)}+o(1/n) \le \frac{2}{n(1-p)}.
\]
Because the parasites $((\mathbf{0},0),i)$ already start at the exit vertex, they have a higher chance to leave $V_n(\mathbf{0})$ before death, and we compute
\[
\begin{split}
\mathbb{P}(\gamma(0,i) < \tau(0,i)) &= \sum_{\ell=0}^n\binom{n}{\ell}p^{n-\ell}(1-p)^\ell\frac{3}{3+n} \sum_{j=0}^\infty\left(\frac{n-\ell}{(n+3)(\ell+1)}\right)^j \\
&= \sum_{t=0}^n\binom{n}{\ell}p^{n-\ell}(1-p)^\ell\frac{3(\ell+1)}{(n+4)\ell+3} \\
&\le p^{n} + \frac{1}{n+4}\sum_{\ell=1}^n\binom{n}{\ell}p^{n-\ell}(1-p)^\ell\frac{3(\ell+1)}{\ell}\\
&\le p^{n} + \frac{6}{n+4},
\end{split}
\]
where in the last line we estimated $\frac{l+1}{l}\le 2$ for $l=1,\dots,n$ and the remaining sum is the probability for a binomial distribution to be at least $1$ and, in particular, is less than $1$. Now $\xi_1$ is exactly given by the amount of parasites which leave $V_n(\mathbf{0})$ before death, and thus
\[
\begin{split}
\mathbb{E}[\xi_1] &= \sum_{i=1}^3\mathbb{P}(\gamma(0,i)<\sigma(0,i)) + \sum_{k=1}^n\sum_{i=1}^2\mathbb{P}(\gamma(k,i)<\sigma(k,i)) \\
&\le 3\left(p^{n} + \frac{6}{n+4}\right) + 2n\left(p^{n} + \frac{3}{n+4}\cdot\frac{2}{n(1-p)}\right) \overset{n\to\infty}{\to} 0.
\end{split}
\]
In particular, $\mathbb{E}[\xi_1] \le 1$ for $n$ large enough, and thus the Bienaymé-Galton-Watson branching process $(\xi_m)_{m\ge1}$ and, consequently, the SIMI die out almost surely. Hence, we must have $p_c(G^{(n)},3) \ge p$ for any $p\in (0,1)$ and $n$ large enough and conclude the proof.
\end{proof}

\begin{proof}[Proof of Theorem \ref{thm:nonmon_graph}]
 Take $G= \mathbb{T}_3$ and note that by Theorem \ref{thm:trivpos2} and Theorem \ref{thm:Tdasympt} there exists some $d_0 \ge 3$ such that
 \[
 p_c(\mathbb{T}_3,3)\ge \frac{1}{2} > p_c(\mathbb{T}_{d_0},3).
 \]
 Also, by Theorem \ref{thm:Tdphase} we have $p_c(\mathbb{T}_3,3)\in(0,1)$, and thus by Proposition \ref{prop:compgraph} there is some $n_0\in\N$ such that
 \[
 p_c(G^{(n_0)},3) > p_c(\mathbb{T}_3,3).
 \]
 Setting $G_2 = \mathbb{T}_{d_0}$ and $G_1 = G^{(n_0)}$ shows the claimed non-monotonic behaviour of the critical value. To show $p_c(G^{(n_0)},3)<1$, we couple the SIMI on $G^{(n_0)}$ with the SIMI on $\mathbb{T}_3$ by having a host at $x\in\mathbb{T}_3$ be susceptible in the SIMI on $\mathbb{T}_3$ if and only if every host inside $V_{n_0}(x)$ is susceptible in the SIMI on $G^{(n_0)}$. This yields $1 > p_c(\mathbb{T}_3,3)^{\frac{1}{n_0+1}} \ge p_c(G^{(n_0)},3)$ and concludes the proof.
\end{proof}

\subsection{Proof on recurrence}
\begin{proof}[Proof of Theorem \ref{thm:norecphase}]
 Again, we follow an approach similar to \cite[Theorem 1.9]{A2002a}. We will show that, with probability $1$, there are only finitely many vertices $\mathcal{H}\subset V$, such that the parasites that will be produced if $v\in\mathcal{H}$ is infected reach $\mathbf{0}$. Clearly, using Lemma \ref{Lem:survdom}, we obtain that replacing the offspring distribution $A$ by $\widetilde{A}$, which is distributed as $A$ conditioned to be at least $1$, will increase the recurrence probability; hence $\vartheta(G,p,A) \le \vartheta(G,p,\widetilde{A})$. For the SIMI with offspring distributed as $\widetilde{A}$, we can then make use of the variables defined in \eqref{eq:perc_coupl} to study the behaviour of the model. We fix some $p\in[0,1)$ and for $x\in V\setminus\{\mathbf{0}\}$ set
  \[
  \mathcal{R}_x^{p} := \bigcup_{i=1}^{\widetilde{A}_x-1}\mathcal{V}^p_{x,i}(\{x\}),
  \]
  the vertices that will be reached by some parasite that was activated at $x$. Also, we define
  \[
  \mathcal{R}_\mathbf{0}^{p} := \bigcup_{i=1}^{\widetilde{A}_\mathbf{0}}\mathcal{V}^p_{\mathbf{0},i}(\{\mathbf{0}\})
  \]the set of vertices visited by the initially activated parasites. Because $G$ is quasi-vertex-transitive, there is a finite partition $V_1,V_2,\dots,V_k\subset V$ such that for any $1\le\ell\le k$ and $x,y\in V_\ell$ there is a graph automorphism $\phi\in \text{Aut}(G)$ with $\phi(x) = y$. In particular, considering the worst orbit to get out of, this implies that there is some constant $C> 0$ such that for any $x\in V$ we have $\mathbb{P}(\mathbf{0}\in\mathcal{R}_x^{p})\le C \mathbb{P}(x\in\mathcal{R}_\mathbf{0}^{p})$.  For every $p\in[0,1)$ we calculate
  \[
  \begin{split}
    \sum_{x\neq \mathbf{0}}\mathbb{P}(\mathbf{0}\in\mathcal{R}_x^{p}) &\le \sum_{x\neq\mathbf{0}}C\mathbb{P}(x\in\mathcal{R}_\mathbf{0}^{p}) = C\mathbb{E}\left[\sum_{x\neq\mathbf{0}}\mathbf{1}_{x\in \mathcal{R}^p_{\mathbf{0}}}\right] =\mathbb{E}\left[\vert\mathcal{R}_\mathbf{0}^{p}\setminus\{\mathbf{0}\}\vert\right]\\
    &=C\mathbb{E}\left[\left\vert\bigcup_{i=1}^{\widetilde{A}_\mathbf{0}}\mathcal{V}_{\mathbf{0},i}^p(\{\mathbf{0}\})\setminus\{\mathbf{0}\}\right\vert\right] \le C \mathbb{E}\left[\sum_{i=1}^{\widetilde{A}_\mathbf{0}}\left\vert\mathcal{V}_{\mathbf{0},i}^p(\{\mathbf{0}\})\setminus\{\mathbf{0}\}\right\vert\right] \\ &= C \mathbb{E}[\widetilde{A}]\mathbb{E}\left[\left\vert\mathcal{V}_{\mathbf{0},1}^p(\{\mathbf{0}\})\right\vert\right] < \infty,
\end{split}
\]
where in the last equality we used the independence of $\widetilde{A}_\mathbf{0}$ from the identically distributed collection $\{\vert\mathcal{V}^p_{\mathbf{0},i}(\{\mathbf{0}\})\vert:i\in\N\}$ and finally used that $\vert\mathcal{V}^p_{\mathbf{0},1}(\{\mathbf{0}\})\vert$ is geometrically distributed with parameter $1-p > 0$ and, in particular, has a finite expectation. By the Borel-Cantelli Lemma there are almost surely only finitely many $x\neq \mathbf{0}$ such that $\mathbf{0}\in\mathcal{R}_x^{p}$. Hence, $\mathbf{0}$ is almost surely only finitely often visited by parasites.
\end{proof}

\section{Some open questions on the SIMI}
\subsection{Monotonicity of \texorpdfstring{$p_c$}{pc} as a function of \texorpdfstring{$A$}{A} under weaker stochastic orders}~

Theorem \ref{thm:nonmon_graph} establishes a non-monotonicity of $p_c$ as a function of the underlying graph.
However, the monotonicity of $p_c$ as a function of the given offspring distribution is also of interest.
A simple answer is obtained if $A^\prime$ stochastically dominates $A$. In this case, we can couple $A_x\le A_x^\prime$ almost surely for $x\in V$ and, thus, invoking a similar argument as Lemma \ref{Lem:survdom}, shows that we have $p_c(G,A) \le p_c(G,A^\prime)$.
However, many interesting cases for offspring distributions are not related in such a strong way, e.g., a deterministic $A=k$ and a Poisson-distributed $A^\prime \sim \text{Pois}(\mu)$ with $\mu > k$.
In \cite{J2018}, similar questions on the frog model are investigated by relating the offspring distributions by a weaker stochastic order. Specifically, the increasing concave order $\mathbb{E}[f(A)] \le \mathbb{E}[f(A^\prime)]$ for all bounded, increasing and concave $f:[0,\infty] \to \mathbb{R}$ and the probability generating function order $\mathbb{E}[t^A] \ge \mathbb{E}[t^{A^\prime}]$ for all $t\in (0,1)$. Their results hold in a quite general framework. However, they do assume the paths of frogs to be independent. In the SIMI this is no longer the case, as parasites can die at the same host, and thus, the paths are only conditionally independent given $\mathbf{I}^p$. Still, \textbf{we conjecture similar results to also hold for the SIMI}, but to prove them, we need to either adapt the approach of \cite{J2018} to apply also for non-independent paths or, to directly apply their results, we need to show our results in a quenched setting, i.e., that for $p$ large enough, there is some event $B\in \sigma(\mathbf{U})$ with positive probability such that for every realisation $\omega\in B$, the SIMI with deterministic immunities given through $\mathbf{I}^p(\omega)$ and random offspring $\mathbf{A}$ and parasite paths $\mathbf{Y}$ has a positive probability to survive.
Comparison results as in \cite{J2018} would also yield the improvement of establishing a phase transition for a broader spectrum of offspring distributions $A$ with $\mathbb{P}(A=0) > 0$ (cf. Theorem \ref{thm:crit_genA_Z}, Theorem \ref{thm:crit_genA_T}).

~
\subsection{Conditions on monotonicity of \texorpdfstring{$p_c$}{pc} as a function of the graph}~

Considering the non-monotonicity of $p_c$ established in Theorem \ref{thm:nonmon_graph} on one hand and the conver\-gence $p_c(\mathbb{T}_d,A)\to 1/\mathbb{E}[A]$ established in Theorem \ref{thm:Tdasympt} suggesting some kind of monotonicity on the other hand, a natural question arises: \textbf{Find conditions for any graphs $G_1,G_2$ that establish an inequality of critical values $p_c(G_1,A) \ge p_c(G_2,A)$}. Results in this direction would also broaden the class of graphs for which we can establish a phase transition. For example, it is clear that for any tree $G$ where every vertex has degree at least $3$, a similar proof as for Theorem \ref{thm:Tdphase} establishes $p_c(G,A) < 1$. However, for any graph $G$ that in particular contains loops, the existence of a phase transition of the SIMI on $G$ seems to require arguments specifically designed to fit the graph. In \cite{F2004} a similar question is already asked for the frog model with death and is, to our knowledge, still unanswered.

~
\subsection{Asymptotic of \texorpdfstring{$p_c(\mathbb{Z}^d,A)$}{pc(Zd,A)}}~

Theorem \ref{thm:Tdasympt} establishes that asymptotically the graph structure of $\mathbb{T}_d$ plays a diminishing role and, in terms of survival, the SIMI behaves more and more like a Bienaymé-Galton-Watson branching process. In \cite[Theorem 1.8]{A2002a} a similar result is also established for the frog model with death on $\Z^d$ for $d\to\infty$. \textbf{We conjecture that also $p_c(\Z^d,A) \to 1/\mathbb{E}[A]$ as $d\to\infty$}. The main problem is that, when trying to couple with a supercritical branching process, even for large (but finite) $d$, the graph structure of $\Z^d$ eventually comes into play by the existence of loops, and the coupling will break. A second coupling with a percolation would still be necessary, which, due to the mutual dependence on the locations of immune hosts, is difficult to show the supercriticality by standard arguments.

~
\subsection{Sharpness of Theorem \ref{thm:trivpos} and Theorem \ref{thm:trivpos2}}~

We were able to establish lower bounds for the critical parameter of survival harnessing, on one hand, an underlying branching process and ignoring the graph structure in Theorem \ref{thm:trivpos} and, on the other hand, relying on the underlying graph structure and ignoring the branching of parasites in Theorem \ref{thm:trivpos2}. A natural question in this setting is if one effect already determines $p_c$ if the other effect diminishes as observed in Theorem \ref{thm:Tdasympt}. This leads to the following questions:
\begin{itemize}
 \item \textbf{Does there, for any graph $G$, exist a distribution $A$ such that $p_c(G,A) = \widetilde{p}_c(G)$?}
 \begin{itemize}
 \item We note here that the question in the other direction is addressed in Theorem \ref{thm:Tdasympt}. Thus, a similar weaker result might be a first step, i.e., showing that for any graph $G$ there exists a sequence $A^{(n)}$ such that $p_c(G,A^{(n)}) \to \widetilde{p}_c(G)$.
 \end{itemize}
 \item A refinement on the first question is if an infinite expectation is already enough to match the critical parameters. \textbf{Does there exist a graph $G$ and a distribution $A$ with $\mathbb{E}[A] = \infty$, but $p_c(G,A) > \widetilde{p}_c(G)$?} Of particular interest is the case that $\widetilde{p}_c(G) = 0$.
\end{itemize}

\section*{Acknowledgments}
I would like to thank the anonymous referee for providing detailed comments and suggestions which helped me to improve this article. The author
was supported by the German Research Foundation (DFG) under Project ID: 443227151.


\bibliographystyle{plain}

\end{document}